\documentclass[11pt]{amsart}

\usepackage{amsmath}
\usepackage{amssymb}
\usepackage{pdfsync}

\newcommand{\R}{{\mathbb R}}       
\newcommand{\Z}{{\mathbb Z}}       
\newcommand{\DD}{{\mathcal D}}
\newcommand{\FF}{{\mathcal F}}
\newcommand{\HH}{{\mathcal H}}

\newcommand{\UU}{{\mathcal U}}

\newcommand{\BZ}{{\mathcal B}}

\newcommand{\RR}{{\mathcal R}}
\newcommand{\NN}{{\mathcal N}}

\newcommand{\TT}{{\mathcal T}}

\newcommand{\diam}{{\rm diam}}
\newcommand{\dist}{{\rm dist}}

\newcommand{\rf}[1]{{(\ref{#1})}}

\newcommand{\supp}{\operatorname{supp}}

\newcommand{\vphi}{{\varphi}}
\newcommand{\ve}{{\varepsilon}}
\newcommand{\vv}{{\vspace{2mm}}}
\newcommand{\vvv}{\vspace{4mm}}
\newcommand{\wt}[1]{{\widetilde{#1}}}

\newcommand{\sss}{{\rm Stop}}
\newcommand{\rest}{{\lfloor}}

\newtheorem{theorem}{Theorem}[section]
\newtheorem{lemma}[theorem]{Lemma}

\newtheorem{propo}[theorem]{Proposition}

\newtheorem*{lemma*}{Lemma}

\theoremstyle{definition}

\theoremstyle{remark}
\newtheorem{rem}[theorem]{\bf Remark}

\numberwithin{equation}{section}

\newcommand{\brem}{\begin{rem}}
\newcommand{\erem}{\end{rem}}

\begin{document}

\title{Uniform measures and uniform rectifiability}
\author{Xavier Tolsa}

\address{Xavier Tolsa. Instituci\'{o} Catalana de Recerca i Estudis Avan\c{c}ats (ICREA) and Departament de
Ma\-te\-m\`a\-ti\-ques, Universitat Aut\`onoma de Bar\-ce\-lo\-na, Catalonia}
\email{xtolsa@mat.uab.cat} 

\thanks{Funded by the European Research
Council under the European Union's Seventh Framework Programme 
ERC Grant agreement 320501 (FP7/2007-2013). Also, partially supported by grants  2014-SGR-75 (Generalitat de Catalunya) and MTM-2010-16232,
MTM-2013-44304-P (MICINN, Spain).} 


\begin{abstract}
In this paper it is shown that if $\mu$ is an $n$-dimensional Ahlfors-David regular measure in
$\R^d$ which satisfies the so-called weak constant density condition, then $\mu$ is uniformly
rectifiable. This had already been proved by David and Semmes in the cases $n=1$, $2$ and $d-1$, and
it was an open problem for other values of $n$. The proof of this result relies on the study
of the $n$-uniform measures in $\R^d$. In particular, it is shown here that they satisfy the ``big pieces of Lipschitz graphs'' property.
\end{abstract}

\maketitle

\section{Introduction}

Given $n>0$, a Borel measure $\mu$ in $\R^d$ is said to be $n$-uniform if there exists some constant $c_0>0$
such that 
\begin{equation}\label{equnif21}
\mu(B(x,r))= c_0\, r^n\qquad\mbox{for all $x\in\supp\mu$ and all $r>0$.}
\end{equation}
For definiteness, let us remark that we assume
the balls $B(x,r)$ to be open.

The study of $n$-uniform measures is a subject of great importance in geometric measure theory because of its applications
to many other problems. In particular, this plays a fundamental role in the proof of the celebrated Preiss' theorem \cite{Pr} 
which
states that, given a Borel measure $\sigma$ on $\R^d$, the $\sigma$-a.e.\ existence of the density 
$$\Theta^n(x,\sigma)=\lim_{r\to 0}\frac{\sigma(B(x,r))}{r^n},$$
with $\Theta^n(x,\sigma)>0$ $\sigma$-a.e.,
forces $\sigma$ to vanish out of a countably $n$-rectifiable subset of $\R^d$. There are many other works 
in areas such as geometric measure theory, potential analysis or PDE's where the knowledge of the structure of $n$-uniform measures is an important ingredient. See, for example, \cite{Ba},
\cite{DKT}, \cite{KPT}, \cite{KT}, or \cite{PTT}.

The classification of $n$-uniform measures in $\R^d$ is
a difficult problem which is solved only partially.
 To begin with, it is easy to check that there are no non-zero $n$-uniform measures for $n>d$. The first remarkable
 result was obtained by Marstrand  \cite{Marstrand}, who showed that if $\mu$ is a non-zero $n$-uniform measure, then $n\in\Z$. Later on, Preiss proved in \cite{Pr} that
 for $n=1$ or $n=2$, any $n$-uniform measure $\mu$ is flat, i.e.\ it is of the form
 $c\,\HH^n\rest L$, where $c$ is some constant, $L$ is an $n$-plane, and $\HH^n\rest L$ stands for the $n$-dimensional Hausdorff measure restricted to $L$.
 
Quite surprisingly, for $n\geq3$, $d\geq4$, there are $n$-uniform measures which are non-flat. This is the case,
for example, of the $3$-dimensional Hausdorff measure restricted to the light cone in $\R^4$, defined by
$$X:=\{x\in\R^4:\,x_4^2= x_1^2+x_2^2+x_3^2\},$$
where $x_i$ is the $i$-th coordinate of $x$. In $\R^d$ with $d\geq 4$, the Cartesian product of $X$ with any
$(n-3)$-dimensional linear subspace is also an $n$-uniform measure.
In fact, in the codimension $1$ case ($n=d-1$) 
Kowalski and Preiss \cite{KoP} proved that the only $n$-uniform measures are either flat measures or the Hausdorff
measure $\HH^n$ restricted to some rotation and translation of $X\times \R^{n-3}$.

Much less information is known in the case when $n\neq 1$, $2$, $d-1$. In this general situation, 
Kirchheim and Preiss
proved in \cite{KiP} that the support of an $n$-uniform measure coincides with an analytic variety, i.e.\
the zero set of some real analytic function, which moreover is given explicitly in terms of $\mu$.
However, from this result it is difficult to derive quantitative information which can be applied
to solve other problems from geometric measure theory, such as, for example, the characterization of uniformly rectifiable measures in terms of the so-called ``weak constant density condition''.

We turn now to the topic of uniform rectifiability.
An $n$-dimensional Ahlfors-David regular (AD-regular, for short) measure $\mu$ in 
$\R^d$ is a Borel measure such that,
for some constant $c_1>0$, 
\begin{equation}\label{eqad1}
c_1^{-1}\,r^n\leq \mu(B(x,r))\leq c_1\,r^n\qquad\mbox{for all $x\in\supp\mu$, $r>0$.}
\end{equation}

The notion of uniform $n$-rectifiability (or simply, uniform rectifiability) was introduced by David 
and Semmes in \cite{DS}. This is a kind of quantitative version of $n$-rectifiability.
One of the many equivalent definitions is the following: $\mu$ is uniformly rectifiable if 
it is AD-regular and there exist constants $\theta,M>0$ so that, for
each $x\in\supp\mu$ and $R>0$, there is a Lipschitz mapping $g$ from the $n$-dimensional ball $B_n(0,r)\subset\R^n$
to $\R^d$ such that $g$ has Lipschitz norm not exceeding $M$ and
$$\mu\bigl(B(x,r)\cap g(B_n(0,r))\bigr) \geq \theta r^n.$$
In the language of \cite{DS}, this means that {\em $\supp\mu$ has big pieces of Lipschitz images of $\R^n$}.

A related and more restrictive  notion is the one of having ``big pieces of Lipschitz graphs'': an  AD-regular measure $\mu$ in
$\R^d$ has big pieces of Lipschitz graphs if there are constants $\theta,M>0$ so that, for
each $x\in\supp\mu$ and $R>0$, there is a Lipschitz function $g:\R^n\to\R^{d-n}$
with Lipschitz norm not exceeding $M$ whose graph or a rotation of this, called $\Gamma$, satisfies
$$\mu\bigl(B(x,r)\cap \Gamma\bigr) \geq \theta r^n.$$
It is immediate to check that if $\mu$ has big pieces of Lipschitz graphs, then
it is uniformly rectifiable. The converse is not true, unless $n=1$ (see \cite[Chapter I.1]{DS}).

In the monographs \cite{DS1} and \cite{DS} David and Semmes obtained many different characterizations of
uniform rectifiability. One of these characterizations is given in terms of the weak constant density condition (WCD, for short) mentioned above. Given $\mu$ satisfying \rf{eqad1}, denote by $G(c_1',\ve
)$ the subset of
those $(x,r)\in \supp\mu\times (0,+\infty)$ for which there exists a Borel measure $\sigma=\sigma_{x,r}$ satisfying
$\supp\sigma=\supp\mu$, the AD-regularity condition \rf{eqad1} with the constant $c_1'$, and
$$\bigl|\sigma(B(y,t)) - t^n\bigr|\leq\ve \,r^n\quad\mbox{for all $y\in\supp\mu\cap B(x,r)$
and all $0<t<r$.}$$
One says that $\mu$ satisfies the WCD if there is $c_1'>0$ such that the set $ G(c_1',\ve)^c :=
\bigl[\supp\mu\times (0,+\infty)\bigr]\setminus G(c_1',\ve)$ is a Carleson set for every 
$\ve>0$. This means that for every $\ve>0$ there is a constant $C(\ve)$  such that
$$\int_0^R\int_{B(x,R)}\chi_{G(c_1',\ve)^c}(y,r)\,d\mu(y)\,\frac{dr}r \leq C(\ve)\,R^n
\quad\mbox{for all $x\in\supp\mu$
and all $R>0$.}$$

It is not difficult to show that if $\mu$ is uniformly rectifiable, then 
it satisfies the WCD (see \cite[Chapter 6]{DS1}). Conversely, it was proved by David and Semmes that, in the cases $n=1$, $2$, and $d-1$, 
if $\mu$ satisfies the WCD, then it is uniformly rectifiable (see \cite[Chapter III.5]{DS}).
Roughly speaking, to prove this implication, by means of a compactness argument they showed that if 
the WCD holds, then $\mu$ can approximated at most locations and scales by $n$-uniform
measures. For $n=1$ and $2$, this implies that $\mu$ is very close to some flat measure
at most locations and scales
since all $n$-uniform measures turn out to be flat in this case. Then, by the so-called
``bilateral weak geometric lemma'' (see Theorem \ref{teobwgl} below for more details), $\mu$ is 
uniformly rectifiable. The proof for $n=d-1$ is more complicated and it relies on the
 precise
characterization of all the $n$-uniform measures by Kowalski and Preiss in the codimension $1$ case, 
which has already been described above. 

The fact that the WCD implies uniform rectifiability was open up to now for $n\neq1,2,d-1$.
In the present paper we solve this problem.

 \begin{theorem}\label{teowcd}
Let $n$ be an integer with $0<n<d$.
If $\mu$ is an $n$-dimensional AD-regular measure in $\R^d$ satisfying the weak constant density condition, then $\mu$ is uniformly rectifiable. Thus, uniform rectifiability is equivalent to satisfying the weak constant density condition, for AD-regular measures.
\end{theorem}

Our proof is also based on the fact that at most locations and scales $\mu$ is well approximated
by $n$-uniform measures, as shown in \cite{DS}. A fundamental ingredient of our arguments
is the fact, proved by Preiss \cite{Pr}, that if an $n$-uniform measure $\sigma$ is ``quite flat'' at infinity,
then $\sigma$ is actually a flat measure. See (b) of Theorem \ref{teopreiss} for the precise statement. 
In a sense, from this result it follows that if $\sigma$ is very flat at some scale and location, then it
will continue to be flat in the same location at all the smaller scales
(see Lemmas \ref{lemstable} and \ref{mthhh} for more details). Let us remark that a
similar argument was already applied in \cite{PTT} to study the so-called H\"older doubling 
measures. Another essential property of any $n$-uniform measure $\sigma$ is that any ball centered in $\supp\sigma$
contains another ball with comparable radius where $\sigma$ is very flat. This will be proved below by ``touching point arguments'', which involve the $n$-dimensional Riesz transforms as an auxiliary tool (see Lemma
\ref{lemtouch}).

Besides the application to the study of the WCD in connection to uniform rectifiability,
the aforementioned results concerning $n$-uniform measures have the following consequence:

\begin{theorem}\label{teounifmeu}
Let $n$ be an integer with $0<n<d$ and let $\mu$ be an $n$-uniform measure in $\R^d$. Then $\mu$ has big pieces of Lipschitz graphs.
In particular, $\mu$ is uniformly rectifiable.
\end{theorem}

Recall that Kirchheim and Preiss \cite{KiP} proved that if $\mu$ is $n$-uniform, then $\supp\mu$
coincides with an analytic variety. 
One might think that then one can easily deduce that $\mu$ has big pieces of Lipschitz graphs.
This is not the case, as far as the author is concerned, due to the difficulty to derive 
quantitative information from the results in \cite{KiP}.

Finally, it is also worth comparing Theorem \ref{teounifmeu} to a connected result of
De Pauw \cite{DeP}, which asserts that if $\mu$ is a Radon measure in $\R^n$ such that the density $\frac{\mu(B(x,r))}{r^n}$ is increasing as a functions of $r$ for all $x\in\R^d$ and it is uniformly bounded
from below for $x\in\supp\mu$, then $\mu$ has big pieces of Lipschitz graphs.

\vvv


\section{Preliminaries}

As usual in harmonic analysis and geometric measure theory, in this paper 
the letter $c$ is used is to denote
a constant (often an absolute constant) which may change at different occurrences and whose
value is not relevant for the arguments. On the other hand, constants with subscripts, such as
$c_0$ or $c_1$, retain their values at different occurrences.

Unless stated otherwise, all the balls denoted by $B$ in this paper are assumed to be open.

\vv

\subsection{The $\beta_\mu$ and $b\beta_\mu$ 
coefficients }

Given a measure $\mu$ and a ball $B=B(x,r)$, we consider the following Jones' $ \beta_\mu$ coefficient of $B$:
$$
 \beta_\mu(B) =  
 \inf_L \biggl(\sup_{x\in \supp\mu\cap B}\frac{\dist(x,L)}{r}\biggr),$$
where the infimum is taken over all $n$-planes $L\subset\R^d$.
The bilateral $\beta_\mu$ coefficient of $B$ is defined as follows:
$$
 b\beta_\mu(B) = \inf_L \biggl(\sup_{x\in \supp\mu\cap B}\frac{\dist(x,L)}{r}
 + \sup_{x\in L\cap B}\frac{\dist(x,\supp\mu)}{r}
 \biggr),
$$
where the infimum is taken over all $n$-planes $L\subset\R^d$ again.


\subsection{Weak convergence of measures}

We say that a sequence of Radon measures $\{\mu_j\}_j$ in $\R^d$ converges to another Radon measure
$\mu$ in $\R^d$ if, for all continuous functions $f$ with compact support in $\R^d$,
$$\lim_{j\to\infty}\int f\,d\mu_j = \int f\,d\mu.$$

We denote by $AD(c_1,\R^d)$ the family of AD-regular measures in $\R^d$ with constant $c_1$, and by $U(c_0,\R^d)$
the family of uniform measures with constant $c_0$.

\begin{lemma}\label{lemconvtriv}
The following holds:
\begin{itemize}
\item[(a)] If $\{\mu_j\}_j$ is  a sequence of measures from $AD(c_1,\R^d)$, then there exists a subsequence
which converges weakly to some Radon measure in $\R^d$.
 
\item[(b)] If a sequence from $AD(c_1,\R^d)$ converges weakly to a Radon measure $\mu$, then either
$\mu=0$ or $\mu\in AD(c_1,\R^d)$.

\item[(c)] If $\{\mu_j\}_j$ is  a sequence of measures from $U(c_0,\R^d)$ such that $0\in\supp\mu_j$
for every $j$, then $\mu\in U(c_0,\R^d)$.

\end{itemize}
\end{lemma}

The proof of this lemma is standard. See for example \cite[Lemma III.5.8]{DS} for (a) and (b). (c) follows from (b) setting $c_1=1$.

Also we have:

\begin{lemma}\label{lemconbeta1}
Let $\{\mu_j\}_j$ be a sequence from $AD(c_1,\R^d)$ which converges weakly to a Radon measure $\mu$.
Then for every ball $B\subset\R^d$ we have:
$$\lim_{j\to\infty}\left(\sup_{x\in B\cap \supp\mu}\dist(x,\supp\mu_j)\right) = 0$$
and 
$$\lim_{j\to\infty}\left(\sup_{x\in B\cap \supp\mu_j}\dist(x,\supp\mu)\right) = 0.$$
\end{lemma}

See \cite[Lemma III.5.9]{DS}.
The following is an easy consequence whose proof is left for the reader.

\begin{lemma}\label{lemconbeta2}
Let $\{\mu_j\}_j$ be a sequence from $AD(c_1,\R^d)$ which converges weakly to a Radon measure $\mu$, and suppose that $\supp\mu_j\cap B\neq\varnothing$ for all $j$ and $\supp\mu\cap B\neq\varnothing$.
Then
\begin{equation}\label{eqd29}
\frac12\,\limsup_{j\to\infty}\beta_{\mu_j}(\tfrac12 B) \,\leq\, \beta_\mu(B)\, \leq \,
2\,\liminf_{j\to\infty}\beta_{\mu_j}(2 B)
\end{equation}
and
\begin{equation}\label{eqd29b}
\frac12\,\limsup_{j\to\infty}b\beta_{\mu_j}(\tfrac12 B) \,\leq\, b\beta_\mu(B)\, \leq \,
2\,\liminf_{j\to\infty}b\beta_{\mu_j}(2 B).
\end{equation}
\end{lemma}

\vv

\subsection{Uniform measures}

Given a Borel map $T:\R^d\to\R^d$, the image measure $T\#\mu$ is defined by
$T\#\mu(E)=\mu(T^{-1}(E))$, for $E\subset \R^d$.
For $x,y\in\R^d$ and $r>0$, we denote $T_{x,r}(y)=(y-x)/r$.

In \cite[Theorem 3.11]{Pr} it is shown that if $\mu$ is an $n$-uniform measure on $\R^d$, then there exists another $n$-uniform measure $\lambda$ such that
$$\lim_{r\to\infty}\frac1{r^n}\,T_{x,r}\#\mu=\lambda\quad \mbox{weakly for all $x\in\R^d$}.$$
One says that $\lambda$ is the tangent measure of $\mu$ at $\infty$. For more information on tangent measures, see \cite{Pr} or \cite[Chapters 14-17]{Mattila}, for instance.

The following result is the classification theorem of uniform
measures due to Preiss.

\vv
\begin{theorem}[\cite{Pr}]\label{teopreiss}
Let $\mu$ be an $n$-uniform measure on $\R^d$. The following holds:
\begin{itemize}
\item[(a)] If $n=1$ or $2$, then $\mu$ is $n$-flat.
\item[(b)] If $n\geq3$, there exists a constant $\tau_0>0$ depending only on $n$ and $d$ such that
if the tangent measure $\lambda$ of $\mu$ at $\infty$  satisfies
\begin{equation}\label{infty}
\beta_\lambda(B(0,1)) \leq \tau_0
\end{equation}
then $\mu$ is $n$-flat.
\end{itemize}
\end{theorem}
\vv

For the reader's convenience, let us remark that the statement (a)
is in Corollary 3.17 of \cite{Pr}.
Regarding (b), it
is not stated explicitly in \cite{Pr}, although it is a straightforward consequence of
\cite[Theorem 3.14 (1)]{Pr} (and the arguments in its proof) and \cite[Corollary 3.16]{Pr}.
See also \cite[Propositions 6.18 and 6.19]{DeL} for more details.


\vv

\subsection{The $\mu$-cubes and uniform rectifiability}

Below we will use the ``dyadic cubes'' associated with an AD-regular measure $\mu$ built by David in \cite[Appendix 1]{David-wavelets} (see also~\cite{Christ} 
for an alternative construction). 
These dyadic cubes are not true cubes. To
distinguish them from the usual cubes, we will call them ``$\mu$-cubes''. 

Given an AD-regular measure $\mu$ in $\R^d$, the properties satisfied by the $\mu$-cubes are
the following. 
For each $j\in\Z$, there exists a family $\DD_j^\mu$ of Borel subsets of $\supp\mu$ (the dyadic $\mu$-cubes of the $j$-th
 generation) such that:
\begin{itemize}
\item[(i)] each $\DD_j^\mu$ is a partition of $\supp\mu$, i.e.\ $\supp\mu=\bigcup_{Q\in \DD_j^\mu} Q$ and $Q\cap Q'=\varnothing$ whenever $Q,Q'\in\DD_j^\mu$ and
$Q\neq Q'$;
\item[(ii)] if $Q\in\DD_j^\mu$ and $Q'\in\DD_k^\mu$ with $k\leq j$, then either $Q\subset Q'$ or $Q\cap Q'=\varnothing$;
\item[(iii)] for all $j\in\Z$ and $Q\in\DD_j^\mu$, we have $2^{-j}\lesssim\diam(Q)\leq2^{-j}$ and $\mu(Q)\approx 2^{-jn}$;
\item[(iv)] if $Q\in\DD_j^\mu$, there exists some point $z_Q\in Q$ (the center of $Q$) such that $\dist(z_Q,\supp\mu\setminus Q)
\gtrsim 2^{-j}$.
\end{itemize}
We denote $\DD^\mu=\bigcup_{j\in\Z}\DD^\mu_j$. Given $Q\in\DD^\mu_j$, the unique $\mu$-cube $Q'\in\DD_{j-1}^\mu$ which contains $Q$ is called the parent of $Q$.
We say that $Q$ is a son of $Q'$. Two cubes which are sons of the same parent are called brothers.
Also, given $Q\in\DD^\mu$, we denote by $\DD^\mu(Q)$ the family of
$\mu$ cubes $P\in\DD^\mu$ which are contained in $Q$. 

For $Q\in \DD_j^\mu$, we define the side length
 of $Q$ as $\ell(Q)=2^{-j}$. Notice that $\ell(Q)\lesssim\diam(Q)\leq \ell(Q)$.
Actually it may happen that a $\mu$-cube $Q$ belongs to $\DD_ j^\mu\cap \DD_k^\mu$ with $j\neq k$, because there may exist $\mu$-cubes
with only one son. In this case, $\ell(Q)$ is not well defined. However this problem can be solved in many ways.
For example, the reader may think that a $\mu$-cube is not only a subset of $\supp\mu$, but a couple $(Q,j)$, where $Q$ is
a subset of $\supp\mu$ and $j\in\Z$ is such that $Q\in\DD_j^\mu$.


Given a $\mu$-cube $Q$, we denote  
$$B_Q=B(z_Q,3\ell(Q)).$$
 Then we define
 $\beta_\mu(Q):=\beta_\mu(B_Q)$ and $b\beta_\mu(Q):=b\beta_\mu(B_Q)$.

One says that a family $\FF\subset\DD^\mu$ is a Carleson family if there exists some constant $c$ such
that
$$\sum_{Q\in\FF:Q\subset R}\mu(Q)\leq c\,\mu(R)\qquad \mbox{for every $R\in\DD^\mu$.}$$

An AD-regular measure $\mu$ is said to satisfy the ``bilateral weak geometric lemma'' if
for each $\ve>0$ the family of $\mu$-cubes $Q\in\DD^\mu$ such that $b\beta_\mu(Q)>\ve$ is a Carleson
family. The following deep result is due to David and Semmes (see \cite[Chapter II.2]{DS}).

\begin{theorem}\label{teobwgl}
Let $\mu$ be an AD-regular measure in $\R^d$. Then $\mu$ is uniformly rectifiable if and only if
it satisfies the bilateral weak geometric lemma.
\end{theorem}

\vv


\section{Uniform measures}

\vv

\subsection{Existence of many balls with small $\beta_\mu$}

In this section we will prove the following.

\begin{lemma}\label{lemtouch}
Let $\mu$ be an $n$-uniform measure in $\R^d$. For every $\ve>0$ there exists some $\tau>0$ such that 
every ball $B$ centered in $\supp\mu$ contains another ball $B'$ also centered in $\supp\mu$ which satisfies
$\beta_\mu(B')\leq\ve$ and $r(B')\geq \tau\,r(B)$. Moreover, $\tau$ only depends on $\ve$, $n$ and $d$.
\end{lemma}

\vv
We will prove this lemma by touching point arguments involving the
Riesz transforms. For $0<s< t$ and a Radon measure $\nu$ in $\R^d$, we consider 
the doubly truncated $n$-dimensional Riesz transform of $\nu$ at $x\in\R^d$, defined as follows:
$$\RR_{r,s}\nu(x) = \int_{r<|x-y|\leq s} \frac{x-y}{|x-y|^{n+1}}\,d\nu(y).$$
Observe that the kernel $ \frac{x-y}{|x-y|^{n+1}}$ is vectorial, and thus $\RR_{r,s}\nu(x)\in \R^d$.

Next we need to prove some auxiliary results.
\vv


\begin{lemma}\label{lemriesz}
Let $\mu$ be an $n$-uniform measure in $\R^d$ satisfying \rf{equnif21}. Let $z_0\in\supp\mu$ and $r>0$. For all $s>r$ and
all $x\in B(z_0,r)\cap\supp\mu$, we have
$$\left|\frac{x-z_0}r\cdot\RR_{r,s}\mu(z_0)\right|\leq c\,c_0.$$
\end{lemma}

The dot ``$\cdot$'' in the preceding inequality denotes the scalar product in $\R^d$.

\begin{proof}
Without loss of generality we assume that $\mu\in U(1,\R^d)$ and that $z_0=0$.
 
For fixed parameters $r,s$ with $0<r<s$, let $\vphi:\R\to\R$ be a non-negative radial $C^\infty$ function such that:
$$\vphi(t) =\left\{
 \begin{array}{ll} 0 & \mbox{if $|t|\leq \dfrac r2$ or $|t|\geq 2s$,}\\
 &\\
\dfrac1{t^n} & \mbox{if $r\leq |t|\leq s$.}
\end{array}\right.
$$
We also ask $\vphi$ to satisfy:
$$|\vphi(t)|\leq \dfrac 1{t^n}\qquad\mbox{and}\qquad
|\vphi'(t)|\leq c\min\left(\dfrac{1}{r^{n+1}}, \,\dfrac{1}{t^{n+1}}\right) \qquad\mbox{for all $t\in\R$.}$$
The precise values of $\vphi(t)$ for $t\in(\frac r2,r)\cup(s,2s)$ do not matter as soon as the preceding
inequalities are fulfilled.

Consider now the function $\rho:\R\to\R$ defined by 
$$\rho(u) = -\int_u^{\infty}\vphi(t)\,dt,$$
so that $\rho'(t)=\vphi(t)$. Finally, for $y\in\R^d$ set
$$\Phi(y) = \rho\bigl(|y|\bigr).$$
Notice that $\Phi$ is a radial $C^\infty$ function which is supported on $\bar B(0,2s)$.

By Taylor's formula, for all $x,y\in\R^d$ we have
\begin{equation}\label{eqmes*}
\Phi(x-y) - \Phi(-y) = x\cdot \nabla \Phi(-y) + \frac12\,x^T\cdot \nabla^2\Phi(\xi_{x,y})\cdot x,
\end{equation}
for some $\xi_{x,y}$ belonging to the segment $[x-y,-y]\subset\R^d$.
Thus, integrating with respect to $\mu$,
\begin{equation}\label{eqdif2}
\Phi*\mu (x) - \Phi*\mu(0) = x\cdot\int \nabla \Phi(-y)\,d\mu(y) + \frac12 \int x^T\cdot \nabla^2\Phi(\xi_{x,y})\cdot x\,d\mu(y).
\end{equation}
Notice that the measurability of $x^T\cdot \nabla^2\Phi(\xi_{x,y})\cdot x$ follows from the identity \rf{eqmes*}.
For $0,x\in\supp\mu$, we have $\Phi*\mu (x) = \Phi*\mu(0)$ because $\Phi$ is radial and $\mu$ is a uniform measure.
Thus the left side of \rf{eqdif2} vanishes. This is the crucial step where the uniformity
of $\mu$ is used in this lemma.

Observe that 
\begin{equation}\label{eq100}
\nabla\Phi(z) = \rho'(|z|)\,\frac z{|z|} = \vphi(|z|)\,\frac z{|z|}.
\end{equation}
 Thus, 
 \begin{align}\label{eq101}
 &\nabla\Phi(z) = \frac z{|z|^{n+1}}\qquad \mbox{for $r\leq |z|\leq s$,}\\ \label{eq102}
&|\nabla\Phi(z)| \leq \frac c{r^{n}}\qquad \mbox{for $|z|\leq r$,}\\ \label{eq103}
&|\nabla\Phi(z)| \leq \frac c{s^{n}} \qquad \mbox{for $ |z|\geq s$.}
\end{align}
From \rf{eq101} we infer that
$$\int_{r<|y|\leq s} \nabla \Phi(-y)\,d\mu(y) = - \RR_{r,s}\mu(0),$$
and so
$$\int \nabla \Phi(-y)\,d\mu(y) = \int_{|y|\leq r} \nabla \Phi(-y)\,d\mu(y) - \RR_{r,s}\mu(0)
+
\int_{|y|> s} \nabla \Phi(-y)\,d\mu(y).
$$
By \rf{eq102} we have
$$\left|\int_{|y|\leq r} \nabla \Phi(-y)\,d\mu(y)\right|\leq \frac c{r^{n}}\,\mu(B(0,r))= c,$$
and by \rf{eq103}, recalling that $\Phi(|y|)=0$ for $|y|\geq2s$,
$$\left|\int_{|y|> s} \nabla \Phi(-y)\,d\mu(y)\right| \leq \frac c{s^{n}}\,\mu(B(0,2s))= c\,2^n.$$

From \rf{eqdif2} and the preceding estimates we deduce that
\begin{equation}\label{eqriesz3}
\left|x\cdot \RR_{r,s}\mu(0)\right| \leq c\,|x| + \frac12 \int \bigl|x^T\cdot \nabla^2\Phi(\xi_{x,y})\cdot x
\bigr|\,d\mu(y).
\end{equation}
We are going now to estimate the last integral above. From \rf{eq100} and the definition of $\vphi$ it
easily follows that
$$|\nabla^2\Phi(z)|\leq c\,\min\left(\frac1{r^{n+1}},\,\frac1{|z|^{n+1}}\right)\qquad
\mbox{for all $z\in\R^d$}.
$$
For $|y|\leq 2r$, we use the estimate
$$|\nabla^2\Phi(\xi_{x,y})|\leq c\,\frac1{r^{n+1}},$$
and for $|y|> 2r$, taking into account that
 $|x|\leq r$ and $\xi_{x,y}\in[x-y,-y]$, we deduce that $|\xi_{x,y}|\approx|y|$, and thus we have
$$|\nabla^2\Phi(\xi_{x,y})|\leq c\,\frac1{|\xi_{x,y}|^{n+1}}\approx \frac1{|y|^{n+1}}.$$
Therefore,
$$\int \bigl|x^T\cdot \nabla^2\Phi(\xi_{x,y})\cdot x
\bigr|\,d\mu(y) \leq c \int_{|y|\leq2r}  \frac{|x|^2}{r^{n+1}}\,d\mu(y) +
c \int_{|y|> 2r} \frac{|x|^2}{|y|^{n+1}}\,\mu(y)\leq \frac{c\,|x|^2}r,$$
where we used the fact that $\mu(B(0,t))=t^n$ for all $t>0$ to estimate the last two integrals.
So by \rf{eqriesz3} we get
$$\left|x\cdot \RR_{r,s}\mu(0)\right| \leq c\,|x| + c\,\frac{|x|^2}r \leq c\,r,$$
which proves the lemma.
\end{proof}

\vv

We need now to introduce some variants of the $\beta_\mu$ coefficients.
Given $0\leq m\leq d$ and a ball $B\subset\R^d$, we denote
$$
 \beta_\mu^{(m)}(B) = \inf_L \biggl(\sup_{x\in \supp\mu\cap B}\frac{\dist(x,L)}{r(B)}\biggr),$$
where the infimum is taken over all $m$-planes $L\subset\R^d$.
So $\beta_\mu(B)=\beta_\mu^{(n)}(B)$.

\vv
\begin{lemma}\label{lemaux1}
Let $\mu$ be a Radon measure in $\R^d$ and let $B$ be a ball centered in  $\supp\mu$. Suppose that there
exist constants $c_1,\kappa>0$ such that
\begin{equation}\label{eqad2}
c_1^{-1}\,r^n\leq \mu(B(x,r))\leq c_1\,r^n\qquad \mbox{for all $x\in B\cap \supp\mu$ \,and\, $\kappa \,r(B)\leq r\leq 2r(B)$.}
\end{equation}
Suppose also that, for some $\ve>0$,
$$\beta_\mu^{(d-1)}(B(x,r))\geq \ve \qquad \mbox{for all $x\in B\cap \supp\mu$ \,and\, $\kappa \,r(B)\leq r\leq 2r(B)$.}$$
For any $M>0$, if $\kappa=\kappa(M,\ve,c_1)$ is small enough, then
there exist $r\in[\kappa\, r(B),\, r(B)]$ and
 points $x,z_0\in B\cap\supp\mu$ with $|x-z_0|< \kappa\,r(B)$
such that
\begin{equation}\label{eqfo95}
\left|\frac{x-z_0}{\kappa\,r(B)}\cdot\RR_{\kappa r(B),r}\mu(z_0)\right|\geq M.
\end{equation}
\end{lemma}

\begin{proof}
Suppose that $\kappa$ is small enough (depending now on $M,c_1,n,d$). From the condition \rf{eqad2}, it is not difficult to check that
then there exists some open ball 
$B'$ centered at  some point from $\frac14B$ with $r(B')\geq c_2\,r(B)$
which does not intersect $\supp\mu$, with $c_2>0$ depending on $c_1,n,d$. We assume that $\kappa\ll c_0$. We dilate $B'$ till we have $\partial B'\cap\supp\mu\neq\varnothing$ while still $B'\cap\supp\mu=\varnothing$, and we keep the same notation $B'$ for the dilated ball. Observe that $r(B')\leq r(B)/4$, because 
otherwise $B'$ would contain the center of $B$, which belongs to $\supp\mu$. This implies that
$B'\subset \frac12 B$.

Let $z_0\in\supp\mu\cap \partial B'$. By the preceding discussion, $z_0\in\frac12 \bar B$.
Let $L$ be the hyperplane which is tangent to $B'$ at $z_0$, and let $U$ be the closed half-space
whose boundary is $L$ and does not contain $B'$. 
Suppose for simplicity that $z_0=0$ and $U=\{y\in\R^d:\,y_d\geq 0\}$, where $y_d$ stands for the $d$-th coordinate of $y$.
Also,  denote $D=\R^d\setminus U$, i.e.\ $D=\{y\in\R^d:\,y_d< 0\}$.

For each $j\geq 0$, let $B_j$ be the closed ball centered in  $z_0=0$ with radius $r(B_j)=\left(\frac{2}\ve\right)^j
\,\kappa\,r(B)$. It is easy to check that there exists some absolute constant $c_3$ such that
\begin{equation}\label{eqdo24}
\dist(y,L)= |y_d|\leq c_3\,\frac{r(B_j)^2}{r(B')}\qquad
\mbox{if $y\in D\cap B_j\setminus B'$.}
\end{equation}
On the other hand, since $\beta_\mu^{(d-1)}(B_j)\geq\ve$, we infer that
there exists some $y\in B_j\cap\supp\mu$ such that
\begin{equation}\label{eqj22}
\dist(y,L)\geq \ve\,r(B_j).
\end{equation}
Thus, if 
$r(B_j) < c_3^{-1}\ve\,r(B')$,
then the points $y\in B_j$ which satisfy \rf{eqj22} are contained in $U$. 
By the condition \rf{eqad2}, it follows that
$$\mu(U\cap B_{j+1}\setminus (B_{j-1}\cup \UU_{\ve r(B_j)/2}(L)) \geq c_1^{-1}\,c(\ve)\,r(B_j)^n,$$
for some constant $c(\ve)>0$ (here the notation $\UU_\delta(A)$ stands for the $\delta$-neighborhood of $A$).
Taking also into account that $y_d\geq0$ for all $y\in U$, for $j\geq 1$ we deduce that
\begin{align}\label{eqdo98}\nonumber
\int_{U\cap B_{j+1}\setminus B_{j-1}} \frac{y_d}{|y|^{n+1}}\,d\mu(y) &\geq 
\mu\bigl(
U\cap B_{j+1}\setminus (B_{j-1}\cup \UU_{\ve r(B_j)/2}(L))\bigr) \,
\frac{\ve \,r(B_j)}{2\,r(B_{j-1})^{n+1}}\\
& \geq c_1^{-1}\,c'(\ve).
\end{align}
Also, by \rf{eqdo24},
\begin{equation}\label{eqdo99}
\left|\int_{D\cap B_j\setminus B_{j-1}} \frac{y_d}{|y|^{n+1}}\,d\mu(y)\right| \leq 
\mu\bigl(
D\cap B_j\setminus B_{j-1}\bigr) \,
\frac{c_3\,r(B_j)^2}{r(B')\,r(B_{j-1})^{n+1}} \leq \frac{c_1\,c(\ve)\,r(B_j)}{r(B')}.
\end{equation}

Choose now an integer $N>1$ such that $r:=r(B_N)\leq r(B')$
and denote by $\RR_{\kappa\,r(B),r}^d\mu$ the $d$-th coordinate of $\RR_{\kappa\,r(B),r}\mu$. We write
\begin{align*}
\RR_{\kappa\,r(B),r}^d\mu(z_0) & = \sum_{j=1}^N \int_{y\in B_j\setminus B_{j-1}}
\frac{y_d}{|y|^{n+1}}\,d\mu(y) \\
& \geq \sum_{j=1}^{N} \int_{U\cap B_{j}\setminus B_{j-1}} \frac{y_d}{|y|^{n+1}}\,d\mu(y) 
-  \sum_{j=1}^N \left|\int_{D\cap B_j\setminus B_{j-1}} \frac{y_d}{|y|^{n+1}}\,d\mu(y)\right|.
\end{align*}
Notice that, by \rf{eqdo98},
$$\sum_{j=1}^{N} \int_{U\cap B_{j}\setminus B_{j-1}} \frac{y_d}{|y|^{n+1}}\,d\mu(y)\geq \frac12
\sum_{j=1}^{N-1} \int_{U\cap B_{j+1}\setminus B_{j-1}} \frac{y_d}{|y|^{n+1}}\,d\mu(y) \geq \frac{c_1^{-1}\,c'(\ve)}2\,(N-1).$$
Here we took into account that all the summands have positive sign. On the other hand, from \rf{eqdo99}
we derive
$$\sum_{j=1}^N \left|\int_{D\cap B_j\setminus B_{j-1}} \frac{y_d}{|y|^{n+1}}\,d\mu(y)\right|
\leq \sum_{j=1}^N \frac{c_1\,c(\ve)\,r(B_j)}{r(B')} \leq c_1\,c_4(\ve).$$
Thus, setting $\vec n := (0,\ldots,0,1)$,
$$\vec n \cdot \RR_{\kappa\,r(B),r}\mu(z_0) =
\RR_{\kappa\,r(B),r}^d\mu(z_0) \geq \frac{c_1^{-1}\,c'(\ve)}2\,(N-1) - c_1\,c_4(\ve).$$
Since $\beta_\mu^{(d-1)}\bigl(B(0,\kappa\,r(B))\bigr)\geq\ve$, there are points (or vectors) $x_{(1)},\ldots,
x_{(d)}\in\supp\mu\cap B(0,\kappa\,r(B))$
which generate $\R^d$. In fact, one can write
$$\vec n = \sum_{i=1}^d a_i\,\frac{x_{(i)}}{\kappa \,r(B)},$$
and one can check that $|a_i|\leq c_5(\ve)$ for every $i$.
Thus, there exists some $x_{(i)}\in\supp\mu\cap B(0,\kappa\,r(B))$ such that
$$\left|\frac{x_{(i)}}{\kappa \,r(B)}\cdot \RR_{\kappa\,r(B),r}\mu(z_0)\right| \geq 
\frac1{d\,c_5(\ve)}\,\left[\frac{c_1^{-1}\,c'(\ve)}2\,(N-1) - c_1\,c_4(\ve)\right].$$
If $N$ is taken big enough (which forces $\kappa$ to be small enough), then \rf{eqfo95} follows.
\end{proof}

\vv

\begin{lemma}\label{lemtouch2}
Let $\mu$ be an $n$-uniform measure in $\R^d$ and let $n<m\leq d$. For any $\ve>0$ there exist constants $\delta,\tau>0$ such that if $B$ is a ball centered in  $\supp\mu$ such that $\beta_\mu^{(m)}(B)\leq \delta$,
then 
 there exists another ball $B'$ also centered in  $\supp\mu$ which satisfies
$\beta_\mu^{(m-1)}(B')\leq\ve$ and $r(B')\geq \tau\,r(B)$. Moreover, $\tau$ and $\delta$ only depend on $\ve$, $n$ and $d$.
\end{lemma}

\begin{proof}
Without lost of generality, we assume that $\mu\in U(1,\R^d)$. 

Let $L$ be a best approximating $m$-plane for $\beta_\mu^{(m)}(B)$ and denote by $\Pi_L$ the orthogonal projection onto $L$.
Consider the image measure $\wt\mu = \Pi_L\#(\mu\rest B)$.
It is easy to check that
$$r^{-n}\leq \wt\mu(B(x,r))\leq c\,r^n\qquad \mbox{for $x\in \dfrac12 B\cap \supp\wt\mu$\; and\;
$\delta\,r(B)\leq r\leq r(B)$.}$$
We claim that for all $z_0\in\frac12 B\cap \supp\wt\mu$ and $r_0,r$ with
$\delta^{1/2}\,r(B)\leq r_0\leq r\leq r(B)$, if $\delta$ is small enough,
\begin{equation}\label{eqriesz43}
\left|\frac{x-z_0}{r_0}\cdot\RR_{r_0,r}\wt \mu(z_0)\right|\leq c \qquad\mbox{for $x\in B(z_0,r_0)\cap\supp\wt
\mu$,}
\end{equation}
where $c$ is some absolute constant. Assuming this for the moment, Lemmas \ref{lemriesz} and \ref{lemaux1} (via the identification $\R^m\equiv L$
and choosing $d=m$) ensure the existence of a ball
$B'\subset \frac12 B$ centered in $\supp\wt\mu$ with $r(B')\geq\kappa \,r(B)$ such that 
$\beta^{(m-1)}_{\wt\mu}(B')\leq
\frac12\,\ve$, assuming $\delta$ small enough (in particular $\delta\leq\kappa$). Of course, $\kappa$ is the constant given by Lemma \ref{lemaux1} with the appropriate values of $c_1$, $\ve$ and $M$ there. Together with the fact that $\supp\mu\cap B\subset \UU_{\delta r(B)}(L)$, this implies that
there exists some $(m-1)$-plane $L'\subset L$ such that
$$\supp\mu\cap B'\subset \UU_{\delta r(B) + \frac\ve2 r(B')}(L').$$
Thus,
$$\beta^{(m-1)}_\mu(B') \leq \frac{\delta r(B) + \frac\ve2\, r(B')}{r(B')} \leq \delta\,\tau^{-1} + \frac12\,\ve
\leq\ve,$$
if $\delta$ is taken small enough.

It remains to prove \rf{eqriesz43}. First we will estimate the difference between $\RR_{r_0,r}\wt\mu(z_0)$ and
$\RR_{r_0,r}\mu(z_1)$, where $z_1\in\supp\mu\cap B$ is such that $\Pi_L(z_1)=z_0$. Denote by $K$ the kernel
of $\RR$. That is, $K(z_0-y) =\frac{z_0-y}{|z_0-y|^{n+1}}$. Also, set 
$A(z_0,r_0,r)= \{y\in\R^d:\,r_0<|y-z_0|\leq r\}$. Observe that
$$\RR_{r_0,r}\wt\mu(z_0) = \int_{A(z_0,r_0,r)} \!\!K(z_0-y)\,d\Pi_L\#\mu(y) =
 \int_{B\cap \Pi_L^{-1}(A(z_0,r_0,r))}\!\! K\bigl(z_0-\Pi_L(y)\bigr)\,d\mu(y).$$
Therefore,
\begin{align*}
\bigl|\RR_{r_0,r}\wt\mu(z_0) \,-\, &\RR_{r_0,r}\mu(z_1)\bigr|\\ &\leq
\int_{B\cap \Pi_L^{-1}(A(z_0,r_0,r))}\!\! \bigl|K\bigl(z_0-\Pi_L(y)\bigr)
- K(z_1-y))\bigr|\,d\mu(y)\\
& +
\left|\int_{B\cap \Pi_L^{-1}(A(z_0,r_0,r))}K\bigl(z_1-y)\,d\mu(y) -
\int_{A(z_1,r_0,r)}K\bigl(z_1-y)\,d\mu(y)\right|\\
& =: S_1 + S_2.
\end{align*}

To estimate $S_1$, observe that every $y$ in its domain of integration satisfies
$$|z_0-y|\geq |z_0-\Pi_L(y)| - |y-\Pi_L(y)|\geq r_0 - \delta\,r(B)\geq \frac12 \,r_0,$$
because $r_0\geq\delta^{1/2}r(B)\gg \delta\,r(B)$.
Also, we write
$$\bigl|(z_0-\Pi_L(y)) - (z_1-y)\bigr| \leq |z_0-z_1| + |\Pi_L(y)-y| \leq 2\delta\,r(B)
\ll |z_0-y|.
$$
Thus,
$$\bigl|K\bigl(z_0-\Pi_L(y)\bigr) - K(z_1-y))\bigr| \leq \frac{c\,\delta\,r(B)}{|z_0-y|^{n+1}}.$$
So we obtain
\begin{align*}
S_1& \leq \int_{B\cap \Pi_L^{-1}(A(z_0,r_0,r))}\frac{c\,\delta\,r(B)}{|z_0-y|^{n+1}}\,d\mu(y)\\
& \leq \int_{|y-z_0|\geq \frac12\,r_0}\frac{c\,\delta\,r(B)}{|z_0-y|^{n+1}}\,d\mu(y)
\leq \frac{c\,\delta\,r(B)}{r_0}\leq 1,
\end{align*}
where in the third inequality we used the fact that $\mu(B(z_0,t))=t^n$ for all $t>0$ and in the last one that $\delta\,r(B)\ll r_0$.

Concerning $S_2$, we have
$$S_2 \leq \int_{B\cap \left[\Pi_L^{-1}(A(z_0,r_0,r)) \Delta A(z_1,r_0,r)\right]}\bigl|K\bigl(z_1-y)\bigr|\,d\mu(y).$$
It is easy to check that, for $\delta$ small enough,
$$\supp\mu\cap B\cap \left[\Pi_L^{-1}(A(z_0,r_0,r)) \Delta A(z_1,r_0,r)\right] \subset
A(z_1,\tfrac12 r_0,\,2r_0) \cup  A(z_1,\tfrac12 r,\,2r).$$
Then we get
$$S_2 \leq  \int_{A(z_1,\tfrac12 r_0,\,2r_0)}\frac1{|z_1-y|^n}\,d\mu(y) + \int_{A(z_1,\tfrac12 r,\,2r)} \frac1{|z_1-y|^n}\,d\mu(y)
\leq c.$$
So we have shown that 
\begin{equation}\label{eqdif391}
\bigl|\RR_{r_0,r}\wt\mu(z_0) - \RR_{r_0,r}\mu(z_1)\bigr|\leq c.
\end{equation}

Let $x\in\supp\wt\mu \cap B(z_0,r_0)$, and take $x_1\in \supp\mu \cap B$ such that $\Pi_L(x_1)=x$. We have
\begin{align}\label{eq3tt}
\left|\frac{x-z_0}{r_0}\cdot\RR_{r_0,r}\wt \mu(z_0)\right|& \leq 
\left|\frac{x-z_0}{r_0}\cdot\bigl[\RR_{r_0,r}\wt \mu(z_0) - \RR_{r_0,r}\mu(z_1)\bigr]\right|\nonumber \\
&\quad 
+ \left|\frac{(x-z_0)-(x_1-z_1)}{r_0}\cdot\RR_{r_0,r}\mu(z_1)\right|\nonumber\\
&\quad + \left|\frac{x_1-z_1}{r_0}\cdot\RR_{r_0,r}\mu(z_1)\right|
.
\end{align}
The first term on the right side is bounded by some constant due to \rf{eqdif391}.
For the second one we take into account that
\begin{equation}\label{eqdaq16}
\bigl|(x-z_0)-(x_1-z_1)\bigr| \leq |x-x_1|+ |z_0-z_1| \leq 2\delta\,r(B)
\end{equation}
and we use the brutal estimate
\begin{align*}
\bigl|\RR_{r_0,r}\mu(z_1)\bigr|& \leq \int_{r_0<|y-z_1|\leq r}|K(z_1-y)|\,d\mu(y)\\
& \leq c\,
\left(1+\log\frac {r}{r_0}\right)\leq c\,\log\frac {r(B)}{r_0} \leq c\,|\log\delta|.
\end{align*}
So recalling that $r_0\geq \delta^{1/2}\,r(B)$ we obtain
$$\left|\frac{(x-z_0)-(x_1-z_1)}{r_0}\cdot\RR_{r_0,r}\mu(z_1)\right|\leq c\,
\frac{\delta\,r(B)}{r_0} \,|\log\delta|\leq c\,
\delta^{1/2}\,|\log\delta|\leq c.$$

To estimate the last term on the right side of \rf{eq3tt} we wish to apply Lemma \ref{lemriesz}. 
Since the assumption that $x_1\in B(z_1,r_0)$ is not guarantied, a direct application of the lemma
is not possible. Anyway this issue does not cause any significant difficulty. Indeed, from \rf{eqdaq16}
it follows that
$$|x_1-z_1|\leq |x-z_0| + 2\delta\,r(B) \leq r_0 + 2\delta^{1/2}\,r_0 \leq 2\,r_0.$$
Then we set
$$\left|\frac{x_1-z_1}{r_0}\cdot\RR_{r_0,r}\mu(z_1)\right|\leq 
2\,\left|\frac{x_1-z_1}{2r_0}\cdot\RR_{2r_0,r}\mu(z_1)\right| +
\left|\frac{x_1-z_1}{r_0}\cdot\RR_{r_0,2r_0}\mu(z_1)\right|.$$
By Lemma \ref{lemriesz} the first summand on the right side is uniformly bounded. The last summand 
does not exceed $|\RR_{r_0,2r_0}\mu(z_1)|$, which is also uniformly bounded.
So the claim \rf{eqriesz43}, and thus the lemma, is proved.
\end{proof}

\vv
\begin{proof}[\bf Proof of Lemma \ref{lemtouch}]
We just have to apply Lemma \ref{lemtouch2} repeatedly. Indeed, since $\beta^{(d)}_\mu(B)=0$,
we infer that there exists some ball $B_1\subset B$ centered in $\supp\mu$ with $r(B_1)\approx r(B)$ such that
$\beta^{(d-1)}_\mu(B_1)\leq \ve_1$. Again, assuming that $m<d-1$ this implies that that there exists
some ball $B_2\subset B_1$ centered in $\supp\mu$ with $r(B_2)\approx r(B_1)$ such that
$\beta^{(d-2)}_\mu(B_1)\leq \ve_2$, and so on. At the end we will find some ball
$B_{d-n}\subset B_{d-n-1}$ centered in $\supp\mu$ with $r(B_{d-n})\approx r(B_{d-n-1})\approx\cdots \approx r(B)$ such that
$\beta^{(n)}_\mu(B_{d-n})\leq \ve_n$. The constant $\ve_n$ can be taken arbitrarily small if the constants
$\ve_1,\ldots,\ve_{n-1}$ are chosen suitably small too.
\end{proof}
\vv

\vv


\subsection{A stability lemma for the $\beta_\mu$ coefficients and some consequences}

\begin{lemma}\label{lemfac1}
Let $\mu$ be an $n$-uniform measure in $\R^d$. For any $\ve>0$ there exists some $\delta>0$ such that
for $x\in\supp\mu$ and $r>0$, 
if $\beta_\mu(B(x,\delta^{-1}r))\leq\delta^2$, then $b\beta_\mu(B(x,r))\leq\ve$. 
Moreover, $\delta$ only depends on $\ve$, $n$ and $d$.
\end{lemma}

\begin{proof}
Suppose that the lemma does not hold. So
there exists some $\ve>0$ and a sequence of $n$-uniform measures $\mu_j\in U(1,\R^d)$ and balls
$B_j$ centered in $\supp\mu_j$ such that
$\beta_{\mu_j}(j\,B_j)\leq \dfrac1{j^2}$ and $b\beta_{\mu_j}(B_j)\geq\ve$. After renormalizing, we may assume that $B_j=B(0,1)$. 
Consider a weak limit $\nu$ of some subsequence of $\{\mu_j\}_j$. Redefining $\{\mu_j\}_j$ if
necessary, we may assume that it converges weakly to $\nu$. Observe that $\nu$ is non-zero (because $0\in\supp\mu_j$ for all $j$) and $n$-uniform.

Notice that for $1\leq k\leq j$, 
$$\beta_{\mu_j}(B(0,k))\leq \frac jk \beta_{\mu_j}(B(0,j))\leq \frac1{jk}\to 0\quad \mbox{ as $j\to\infty$.}$$
Then it follows from Lemma \ref{lemconbeta1} that $\beta_\nu(B(0,\frac12k))=0$ for all $k\geq1$. Thus, $\nu$ is supported on some $n$-plane.  Then by Theorem \ref{teopreiss} (b) it turns out that $\nu$ is flat.

However, from the fact that $b\beta_{\mu_j}(B(0,1))\geq\ve$ for all $j\geq1$, it follows that
$b\beta_{\nu}(B(0,2))\geq\frac12\ve$ too, and thus $\nu$ is not flat. So we get a contradiction.
\end{proof}

\vv
 Next lemma can be understood
as some kind of stability condition for the $\beta_\mu$ coefficients. This is a variant of Lemma 4.5 from
\cite{PTT}, with a very similar (and simple) proof. For the reader's convenience we will show the detailed
arguments.

\begin{lemma}[Stability lemma]\label{lemstable}
Let $\mu$ be an $n$-uniform measure in $\R^d$ and let $\ve>0$. 
 There exists some constant
$\delta_0$ depending only on $n$ and $d$ and
an integer $N>0$ depending only on 
$\ve$, $n$ and $d$ such that if $B$ is a ball centered in  $\supp\mu$ satisfying 
\begin{equation}\label{eql79}
\beta_{\mu}(2^kB)\leq \delta_0\qquad \mbox{for $1\leq k\leq N$},\ \
\hbox{then }\ \ b\beta_{\mu}(B)\leq\ve.
\end{equation}
\end{lemma}

\vv
A key point in this lemma is that $\delta_0$ does not depend on $\ve$. Indeed, in order to guaranty
$b\beta(B)$ small enough it suffices to take a sufficiently big $N$.

\vv

\begin{proof}
It is enough to show that $\beta_{\mu}(B)\leq\ve$. The full lemma
follows by combining this partial result with Lemma \ref{lemfac1} and adjusting appropriately the parameters
$\ve$ and $N$.

Suppose that the integer $N$ does not exist. Then there is a sequence of $n$-uniform measures on
$\R^d$, $\{\mu_j\}_{j\geq1}$,
and balls $B_j=B(x_j,r_j)$ centered in  $\supp\mu_j$ such that
for any $j$
$$\beta_{\mu_j}(2^kB_j)\leq \delta_0\quad \mbox{for
$1\leq k\leq j$,\;\; but}\quad \beta_{\mu_j}(B_j)>\ve.$$

For each $j\geq1$, consider the renormalized
measure $\wt \mu_j$ defined by
$$
\wt \mu_j(A) = \frac{\mu_j(r_jA + x_j)}{\mu(B_j)}.
$$
That is, $\wt\mu_j=T_{x_j,r_j}\#\mu_j$, where $T_{x_j,r_j}$ is a homothety such that $T_{x_j,r_j}(B_j)
=B(0,1)$. Notice that $\wt\mu_j\in U(1,\R^d)$ and $0\in\supp\wt\mu_j$ for every $j$.
Extracting a subsequence if necessary, we may assume that
$\{\wt \mu_j\}$ converges weakly to another measure $\nu$, which is 
$n$-uniform. Observe that, by \rf{eqd29},
\begin{equation}\label{claim2}
\beta_\nu(B(0,2))\, \geq \,\frac12
\,\limsup_{j\to\infty}\beta_{\wt \mu_j}( B(0,1)) =
\frac12\,\limsup_{j\to\infty} \beta_{\mu_j}(B_j)\geq \frac12\,\ve,
\end{equation}
and, for all $k\geq0$,
\begin{equation}\label{claim1}
\beta_\nu(B(0,2^k))\, \leq \,
2\,\liminf_{j\to\infty}\beta_{\wt \mu_j}(B(0,2^{k+1})) =
2\,\liminf_{j\to\infty} \beta_{\mu_j}(2^{k+1} B_j)\leq 2\delta_0.
\end{equation}
By \rf{eqd29}, the latter estimate implies that the
 tangent measure $\lambda$ of $\nu$ at $\infty$
satisfies
$$\beta_\lambda(B(0,1))\leq 2 \liminf_{k\to\infty} \beta_\nu(B(0,2^k))\leq 4\,\delta_0.$$
Thus, if we assume that $\delta_0\leq\tau_0/4$ (where $\tau_0$ is the constant in \rf{infty}), 
then $\nu$ is flat, by
Theorem \ref{teopreiss}. This contradicts \rf{claim2}, and the lemma
follows.
\end{proof}

\vv

An easy consequence of the preceding lemma is the following.

\begin{lemma}\label{mthhh}
Let $\mu$ be an $n$-uniform measure in $\R^d$.
For any $\eta>0$, there exists $\delta>0$ depending only on $\eta$, $n$ and $d$
 such that if $B$ is a ball centered in  $\supp\mu$ with
$\beta_{\mu}(B)\leq\delta$, then $b\beta_{\mu}(B')\leq\eta$
for any ball $B'\subset \frac12B$ centered in  $\supp\mu$.
\end{lemma}
\vv

The arguments to prove this lemma are very similar to the ones of Theorem 4.2 from \cite{PTT}. Also, this is a particular case of the more general result that will be shown below in 
the forthcoming Lemma \ref{lem10}. For these reasons, we will skip the proof.

\vv

\begin{proof}[\bf Proof of Theorem \ref{teounifmeu}]
Clearly we may assume that $\mu\in U(1,\R^d)$.
Let $\ve>0$.
From Lemmas \ref{lemtouch} and \ref{mthhh}, if follows that there exists some constant $c(\ve)>0$ 
such that any ball $B$ centered in  $\supp\mu$ contains another ball
$B'$ with $r(B')\geq c(\ve)r(B)$ centered in  $\supp\mu$ such that all the balls $B''$ contained in 
$\frac12B'$
and centered in $\supp\mu$ 
satisfy $b\beta_\mu(B'')\leq\ve$.

Consider a $\mu$-cube  $Q\in\DD^\mu$ such that $Q\subset \frac12B'$ and 
$\ell(Q)\approx r(B')$. Then
$\mu$ is  locally $c\ve$-Reifenberg flat on $Q$. That is, $b\beta_\mu (P)\leq c\,\ve$ for any
$\mu$-cube $P\in\DD^\mu$ contained in $Q$. This implies that if $\ve$ is small enough, then $\mu\rest Q$ has big
pieces of Lipschitz graphs (see Theorem 15.2 of \cite{DT}). In particular, there exists
an $n$-dimensional (possibly rotated) Lipschitz graph $\Gamma$ such that 
$$\mu(B\cap\Gamma)\geq\mu(Q\cap\Gamma)\geq\tau\,\mu(Q) \geq \tau' \,\mu(B),$$ with $\tau,\tau'>0$, 
and the bound on the slope of $\Gamma$ depending on 
$\ve,n,d$ only. 
\end{proof}

\vv

\begin{rem}
In fact, from Theorem 1.9 of \cite{PTT} it follows that the graph $\Gamma$ above is the graph of
a $C^{1+\alpha}$ function, if $\ve$ is small enough. On the other hand, since $\supp\mu$ is an analytic
variety by \cite{KiP}, it seems natural to expect $\Gamma$ to be the graph of a real analytic function.
\end{rem}
\vvv


\section{The weak constant density condition implies uniform rectifiability}

Throughout all this section $\mu$ will be an $n$-dimensional AD-regular measure in $\R^d$.

Given a ball $B$ and two Radon measures $\nu$ and $\sigma$ such that $B\cap\supp\nu\neq \varnothing $ and
$B\cap\supp\sigma\neq \varnothing $, we denote
$$d_B(\nu,\sigma) = \sup_{x\in B\cap\supp\nu} \dist(x,\supp\sigma)
+ \sup_{x\in B\cap\supp\sigma} \dist(x,\supp\nu).$$
Given some small $\ve>0$, we denote by $\NN_0(\ve)$ the family of balls $B\subset\R^d$ 
such that there exists
an $n$-uniform measure $\sigma$ in $\R^d$ satisfying
$$d_B(\mu,\sigma)\leq \ve\,r(B).$$
Further, we let $\NN(\ve)$ be the
set of $\mu$-cubes $Q \in\DD^\mu$ such that $B_Q\in\NN_0(\ve)$.

In \cite[Chapter III.5]{DS} the following is proved:

\begin{propo}\label{popofac}
If $\mu$ satisfies the weak constant density condition, then $\DD^\mu\setminus\NN(\ve)$ is a Carleson
family for all $\ve>0$.
\end{propo}

The next lemma is a simple consequence of Lemmas \ref{lemtouch} and \ref{lemfac1}, and the definition of $\NN(\ve)$.

\begin{lemma}\label{lem9}
For all $\eta>0$ there are constants $\ve,\tau>0$ such that if $Q\in\NN(\ve)$, then there exists some
$\mu$-cube $Q'\in\DD^\mu$ with $Q'\subset Q$ such that $b\beta_\mu(Q')\leq\eta$ and $\ell(Q')\geq\tau\,\ell(Q)$.
\end{lemma}

\begin{proof}
Suppose $\ve\ll1$ and let $\sigma$ be an $n$-uniform measure
in $\R^d$ such that 
$$d_{B_Q}(\mu,\,\sigma)\leq \ve\,r(B_Q)=3\ve\ell(Q).$$
Consider $x\in\supp\sigma$ such that $|x-z_Q|\leq3\ve\ell(Q)$ (recall that $z_Q$ stands for the center of $Q$). 
By Lemmas \ref{lemtouch} and \ref{lemfac1} there exists some ball $B=B(x,r)$ such that $b\beta_\sigma(B)\leq \eta$, 
with $B\cap \supp\mu\subset Q$ 
and $r\approx\diam(Q)$ (with the comparability constant depending on $\eta$).
We assume $\ve$ small enough so that $z_Q\in \frac12 B$.
Then we deduce 
\begin{align*}
b\beta_\mu(B(z_Q,\tfrac12r))& \leq c\,b\beta_\sigma(B(z_Q,r)) + c\,\frac{\dist_{B(z_Q,r)}(\mu,\sigma)}r\\
&\leq c\,\eta + c\,\frac{\ve\,\ell(Q)}{r}\leq c\,\eta + c(\eta)\,\ve.
\end{align*}
So if $\ve$ is assumed to be small enough (for $\eta$ fixed), then $b\beta_\mu(B(x_Q,\frac12r))\leq c\,\eta$.

If we take a $\mu$-cube $Q'\in\DD^\mu$ such that $B_{Q'}\subset B(z_Q,\frac12r)$ with $\ell(Q')\approx r$, we have
$$b\beta_\mu(Q')\leq c\,b\beta_\mu(B(z_Q,\tfrac12r))\leq c\,\eta.$$
\end{proof}

\vv

\begin{lemma}\label{lem10}
For all $\eta>0$ there are constants $\ve>0$ and $\delta_1>0$ (both small enough depending on $\eta$) such that if, for a given $k\geq0$, $B$ is a ball centered in  $\supp\mu$ with $b\beta_\mu(B)\leq \delta_1$ such that 
$$2^{-j}B\in\NN_0(\ve)\qquad\mbox{for all\, $0\leq j\leq k$},$$
then $b\beta_\mu(2^{-k}B)\leq \eta$.
\end{lemma}

\begin{proof} 
Let $\delta_0$ be as in Lemma \ref{lemstable} and set
$\ve_0=\frac14\,\min\left(\frac{\delta_0}{4},\,\eta\right)$. Consider the integer $N=N(\ve_0)>0$
given by Lemma \ref{lemstable} (with $\ve$ replaced by $\ve_0$ in the statement of that lemma).

For $j\geq0$, we denote $B_j=2^{-j}B$. We will prove by induction on $j$ that 
$$b\beta_\mu(B_j)\leq \min\left(\frac{\delta_0}4,\,\eta\right)\quad\mbox{ for $0\leq j\leq k$.}$$
For $0\leq j \leq N+2$ this follows easily if $\delta_1$ is assumed to be small enough. Indeed,
we just write
$$b\beta_\mu(B_j)\leq \frac{r(B_0)}{r(B_j)}\, b\beta_\mu(B_0)\leq 2^{N+2}\,\delta_1
\leq \min\left(\frac{\delta_0}4,\,\eta\right),$$
assuming that $\delta_1\leq  2^{-N-2}\,\min\left(\frac{\delta_0}4,\,\eta\right)$ for the last inequality.

Suppose now that 
\begin{equation}\label{eqprop3}
b\beta_\mu(B_{j-N-2}),\,\ldots,\,b\beta_\mu(B_j)\leq \min\left(\frac{\delta_0}4,\,\eta\right),
\end{equation}
and let us see that $$b\beta_\mu(B_{j+1})\leq\min\left(\frac{\delta_0}4,\,\eta\right)$$ too.
Let $\sigma_{j-N-2}$ be an $n$-uniform measure such that 
\begin{equation}\label{eqprop4}
d_{B_{j-N-2}}(\mu,\sigma_{j-N-2})\leq \ve\,r(B_{j-N-2}).
\end{equation}
Assuming $\ve$ small enough (depending on $\delta_0,\eta,N$), from \rf{eqprop3} and \rf{eqprop4} we infer that
$$b\beta_{\sigma_{j-N-2}}(B_{j-N-1}),\,\ldots,\,b\beta_{\sigma_{j-N-2}}(B_{j-1})\leq \delta_0,$$
and thus by Lemma \ref{lemstable}, 
$$b\beta_{\sigma_{j-N-2}}(B_{j})\leq \ve_0=\frac14\,\min\left(\frac{\delta_0}{4},\,\eta\right).$$
Assuming $\ve$ small enough again, together with the condition \rf{eqprop4} this implies that
$$b\beta_\mu(B_{j+1})\leq \min\left(\frac{\delta_0}4,\,\eta\right).$$
\end{proof}
\vv

Next result is the analogous of the preceding one with balls replaced by $\mu$-cubes. 

\begin{lemma}\label{lem11}
For all $\eta>0$ there are constants $\ve>0$ and $\delta_1>0$ (both small enough depending on $\eta$) such that if $P
\in\DD^\mu$ with $P\subset Q$ is such that
$$S\in\NN(\ve)\qquad\mbox{for all $S\in\DD^\mu$ with $P\subset S\subset Q$}$$
and moreover $b\beta_\mu(Q)\leq\delta_1$,
then $b\beta_\mu(P)\leq \eta$.
\end{lemma}

The proof follows easily from Lemma \ref{lem10} and is left for the reader.
\vv

\begin{proof}[\bf Proof of Theorem \ref{teowcd}]
By Theorem \ref{teobwgl}, $\mu$ is uniformly rectifiable if and only if, for all $\eta>0$, the family 
$\BZ(\eta)$ of the $\mu$-cubes $Q\in\DD^\mu$ such that $b\beta_\mu(Q)>\eta$ is a Carleson family.
We will prove that the latter condition holds by using Lemmas \ref{lem9} and \ref{lem11}.

For a fixed $\eta>0$, let $\ve_1>0$ be the constant named $\ve$ in Lemma \ref{lem11}, and consider the constant $\delta_1$ given in that lemma. Let now $\ve_2>0$ be the constant named $\ve$ in Lemma
\ref{lem9} with $\eta$ replaced by $\delta_1$ there. Set
$$\ve=\min(\ve_1,\ve_2).$$

Consider a $\mu$-cube $R\in\DD^\mu$.
We are going to split the family of $\mu$-cubes from $\NN(\ve)$ which are contained in $R$ into disjoint
subfamilies which we will call ``trees''. A collection of $\mu$-cubes $\TT\subset\DD^\mu$ is a tree if it 
verifies the following properties:
\begin{itemize}
\item $\TT$ has a maximal element (with respect to inclusion) $Q(\TT)$ which contains all the other
elements of $\TT$ as subsets of $\R^d$. The $\mu$-cube $Q(\TT)$ is the ``root'' of $\TT$.

\item If $Q,Q'$ belong to $\TT$ and $Q\subset Q'$, then any $\mu$-cube $P\in\DD^\mu$ such that $Q\subset P\subset
Q'$ also belongs to $\TT$.

\item If $Q\in\TT$, then either all the sons belong to $\TT$ or none of them do.
\end{itemize}
We denote by $\sss(\TT)$ the (possibly empty) family of $\mu$-cubes from $\TT$ whose sons do not 
belong to $\TT$.

We proceed now to describe the algorithm for the construction of the family of trees  $\TT_i$,
$i\in I$. Let $Q_1$ the a cube from $\DD^\mu(R)\cap\NN(\ve)$ with maximal side length. This will be the 
root of the first tree $\TT_1$, which is defined recursively by the next rules:
\begin{itemize}
\item $Q_1\in\TT_1$,
\item If $P\in\TT_1$ and all its sons belong to $\NN(\ve)$, then all of them belong to $\TT_1$ too.
\end{itemize}

Suppose now that $\TT_1,\ldots,\TT_i$ have already been defined. Consider now a cube $Q_{i+1}$ with
maximal side length from $\DD^\mu(R)\setminus \bigcup_{1\leq j\leq i}\TT_j$ (if it exists). This is the root from the next tree $\TT_{i+1}$, which is defined recursively by the same rules as $\TT_1$, replacing $Q_1$ by $Q_{i+1}$ and $\TT_1$ by $\TT_{i+1}$.

The family of trees $\TT_i$, $i\in I$, constructed above satisfies the following properties:
\begin{itemize}
\item $\NN(\ve)\cap\DD^\mu(R) = \bigcup_{i\in I} \TT_i$, and $\TT_i\cap\TT_j=\varnothing$ if $i\neq j$;

\item if $Q\in\TT_i$ and one son of $Q$ lies out of $N(\ve)$, then no son of $Q$ is in $\TT_i$;

\item if $Q_i\equiv Q(\TT_i)$ is the root of $\TT_i$, then either some parent or some brother of $Q(\TT_i)$ does not belong to $\NN(\ve)\cap\DD^\mu(R)$. In any case we denote  by $pb(Q(\TT_i))$ this parent or brother of $Q(\TT_i)$.

\end{itemize}

From the latter condition, recalling that $\DD^\mu\setminus\NN(\ve)$ is a Carleson family and using
that each $\mu$-cube $Q\in\DD^\mu$ has a bounded number of sons,
it follows that 
\begin{equation}\label{eqdak2}
\sum_{i\in I} \mu(Q(\TT_i))\leq c\sum_{i\in I} \mu\bigl(pb(Q(\TT_i))\bigr) \leq c\,
\Bigl(\mu(R) + \sum_{Q\in \DD^\mu(R)\setminus \NN(\ve)} \mu(Q)\Bigr) \leq c\,\mu(R).
\end{equation}

We have
\begin{align*}
\sum_{Q\in\DD^\mu(R)\cap\BZ(\eta)}\mu(Q) & \leq 
\sum_{Q\in\DD^\mu(R)\setminus \NN(\ve)}\mu(Q)  + \sum_{Q\in\DD^\mu(R)\cap \NN(\ve)\cap\BZ(\eta)}\mu(Q)\\
&
\leq c\,\mu(R) + \sum_{i\in I} \sum_{Q\in \TT_i\cap \BZ(\eta)}\mu(Q).
\end{align*}
So from \rf{eqdak2} we see that the theorem will be proved if we show that
\begin{equation}\label{eqti1}
\sum_{Q\in \TT_i\cap \BZ(\eta)}\mu(Q) \leq c\,\mu(Q(\TT_i))\qquad \mbox{for each $i\in I$.}
\end{equation}

Given a fixed tree $\TT_i$, consider the family $\FF_i$ (which may be empty) of $\mu$-cubes $P\in\TT_i$ with $b\beta_\mu(P)\leq
\delta_1$  which are maximal with respect to inclusion, where $\delta_1$ is given by Lemma \ref{lem11}.
This lemma ensures that $b\beta(Q)\leq \eta$ if $Q\in\TT_i$ is contained in some $\mu$-cube $P\in\FF_i$.
In other words, if we denote by $\HH_i$ the
$\mu$-cubes $Q\in\TT_i$ which are not contained in any $\mu$-cube $P\in\FF_i$, we have
$\TT_i\cap \BZ(\eta)\subset\HH_i$. Thus
\begin{equation}\label{eqdah22}
\sum_{Q\in \TT_i\cap \BZ(\eta)} \mu(Q) = \sum_{Q\in \HH_i} \mu(Q).
\end{equation}

To each $\mu$-cube $Q\in \HH_i$ we assign a $\mu$-cube $P\in\FF_i\cup\sss(\TT_i)$ contained in $Q$ with maximal side length
(if $P$ is not unique, the choice does not matter), and we set $P=f(Q)$. By Lemma \ref{lem9},
$\ell(P)\geq\tau\,\ell(Q)$.
Therefore, the number of $\mu$-cubes $Q$ such that $P=f(Q)$ for a fixed $\mu$-cube $P$ is bounded above (by some constant depending on $\tau$, $n$, $d$, and the AD-regularity constant of $\mu$). Thus
\begin{align*}
\sum_{Q\in \HH_i} \mu(Q) & \leq c\sum_{Q\in \HH_i} \mu(f(Q)) \leq c\sum_{P\in \FF_i\cup
\sss(\TT_i)} \mu(P)\\
&\leq c\sum_{P\in \FF_i} \mu(P) + c\sum_{P\in \sss(\TT_i)} \mu(P)\leq c\,\mu(Q(\TT_i)),
\end{align*}
taking into account for the last inequality that both $\FF_i$ and $\sss(\TT_i)$ are families whose elements are pairwise
disjoint $\mu$-cubes (within each family). Together with \rf{eqdah22}, this gives \rf{eqti1} and proves 
the theorem.
\end{proof}

\vvv
\vv

Acknowledgement: I would like to thank Vasileios Chousionis for his careful reading of this paper, finding several typos, and making some suggestions.
\vvv

\vvv


\end{document}